\documentclass[11pt,amsfonts]{article}
\usepackage{amsmath,amsfonts,amsthm,amssymb,color}
\usepackage[T1]{fontenc}
\usepackage{enumerate}
\usepackage{graphicx}
\usepackage{float}
\usepackage[ruled,vlined]{algorithm2e}
\usepackage{amssymb}
\usepackage{amsmath}
\usepackage{color}
\usepackage{layout}
\usepackage{amsfonts}
\usepackage{dsfont}
\usepackage{cite}
\usepackage{color}
\usepackage{listings}
\usepackage{anysize}
\usepackage[noend]{algpseudocode}
\usepackage{ulem}
\usepackage{latexsym}

\usepackage{pdfsync}

\newtheorem{prop}{Proposition}
\newtheorem{lemma}{Lemma}
\newtheorem{definition}{Definition}

\newtheorem{theorem}{Theorem}
\newtheorem{remark}{Remark}

\def\real{{\mathord{{\rm I\kern-2.8pt R}}}}        
\def\inte{{\mathord{{\rm I\kern-2.8pt N}}}}

\def\sZZ{{\rm Z\kern-2.8ptem{}Z}}

\def\z{{\mathchoice
  {\sZZ}
  {\sZZ}
  {\rm Z\kern-0.30em{}Z}
  {\rm Z\kern-0.25em{}Z} }}
\def\sQQ{{\kern 0.27em \vrule height1.45ex width0.03em depth0em
          \kern-0.30em \rm Q}}
\def\qu{{\mathchoice
    {\sQQ}
    {\sQQ}
  {\kern 0.225em \vrule height1.05ex width0.025em depth0em \kern-0.25em \rm Q}
  {\kern 0.180em \vrule height0.78ex width0.020em depth0em \kern-0.20em \rm Q}
        }}
\def\sCC{{\kern 0.27em \vrule height1.45ex width0.03em depth0em
          \kern-0.30em \rm C}}
\def\complex{{\mathchoice
    {\sCC}
    {\sCC}
  {\kern 0.225em \vrule height1.05ex width0.025em depth0em \kern-0.25em \rm C}
  {\kern 0.180em \vrule height0.78ex width0.020em depth0em \kern-0.20em \rm C}
        }}

\newcommand{\ba}{\begin{array}}
\newcommand{\ea}{\end{array}}
\newcommand{\be}{\begin{equation}}
\newcommand{\ee}{\end{equation}}
\newcommand{\bea}{\begin{eqnarray}}
\newcommand{\eea}{\end{eqnarray}}
\newcommand{\beaa}{\begin{eqnarray*}}
\newcommand{\eeaa}{\end{eqnarray*}}

\def\z{\zeta}

\font\tenmath=msbm10 \font\sevenmath=msbm7 \font\fivemath=msbm5
\newfam\mathfam \textfont\mathfam=\tenmath
\scriptfont\mathfam=\sevenmath \scriptscriptfont\mathfam=\fivemath

\def \={{\buildrel {\rm (law)} \over =}}

\newcommand{\basa}{\begin{assumption}}
\newcommand{\easa}{\end{assumption}}

\newcommand{\bas}{\begin{assum}}
\newcommand{\eas}{\end{assum}}

\newcommand{\ignore}[1]{}
\begin{document}

\renewcommand{\thefootnote}{\fnsymbol{footnote}}

\title{Generalized Hermite process: tempering, properties and applications}

\author{ H\'ector Araya$^{1,2}$  \\
\small $^{1}$ Facultad de Ingeniería y Ciencias, Universidad Adolfo Ibañez,\\
  hector.araya@uai.cl\\ 
\small  $^{2}$ Data Observatory Foundation, Chile
}

\maketitle
\begin{abstract}
In this work, we introduce a new process by modifying the  kernel  in the time domain representation of the generalized Hermite process. This modification is constructed by means of multiplication of the kernel in the time definition of the process by an exponential tempering factor $\lambda >0$ such that this new process is well defined. Several  properties of the process are studied and an application to non-parametric regression is also given.  
\end{abstract}
\vskip0.3cm

{\bf 2010 AMS Classification Numbers:} 
\vskip0.3cm
{\bf Key Words and Phrases}:  Non Gaussian; Generalized Hermite process; Tempering; Wiener chaos; Limit theorem. 
\section{Introduction}

An interesting, non-Gaussian, extension of fractional Brownian motion (fBm) is the so called Hermite process \cite{tudor-1, pip}. This process can be defined as an iterated Wiener-It\^o integral \cite{major}, and share many of the properties of fBm, such as: self-similarity, stationarity of the increments, regularity of the paths, covariance structure, among others. Since the Hermite process is a non-Gaussian, self-similar with stationary increments process,  it can be a good candidate to be an input in models where self-similarity is observed in empirical data which appears to be non-Gaussian \cite{law, lu}. \\     
Recently, a new process has been introduced in \cite{bai}. This process, called generalized Hermite process, replace the kernel of the Hermite process by  some general kernel. In fact, the process is defined as $Z(t) = I_{k}(h_t)$, where $I_{k}(\cdot)$ denotes de k-tuple  Wiener -It\^o integral, and
$$h_{t}(x_1, \ldots , x_k) :=\int_{0}^{t}  g(s-x_{1}, \ldots , s-x_{k})  1_{\{ s>x_{1}, \ldots , s>x_{k} \}} ds $$
with a suitable homogeneous function $g$ called the generalized Hermite kernel (see Section 2 below for details).

In this work, using the generalized Hermite process introduced above, we construct  a new process by the modification of the kernel in the time representation of the generalized Hermite process introduced in \cite{bai}. For this process we study several properties, such as, stationarity of the increments, covariance structure, scaling properties, among others. Also, we study the special case of the Hermite kernel introduced in \cite{sab} and the filtered version of this same kernel. For these representations, in addition to the properties mentioned above, in the case of the second chaos (Rosenblatt case), we provide a formula for the cumulants of the processes and we show their behavior in terms of the cumulants, as the tempering factor goes to the critical value zero. Finally, a non-parametric estimation application is performed. Precisely, we study the problem of non parametric estimation in a co-integrated model where the noise is a special type of tempered generalized Hermite process.

Tempered processes are defined, in general,  by exponentially tempering the power law kernel in the moving average representation of different types of processes. One of the first cases of this kind of process is the  tempered fractional Brownian motion (TFBM) introduced in \cite{mer-0}. In this paper, the authors  defined the TFBM and study different properties, such as; stationarity of the increments, scaling, the definition of the tempered fractional Gaussian noise (TFGN), among others. Since then, a numerous amount of references related to the TFBM can be found in the literature, and although we do not attempt to do a complete review on the topic, we only mention a few references  (see \cite{chen, mer-0, mer-1, mer-2 , sab-1, sab-2} and the reference therein).

With respect to the case of tempered non - Gaussian processes related to long memory there are  significantly less references. In  \cite{sab} the author  analyze the case of the tempered Hermite process, give the properties of the process and construct a weak approximation by means of a certain discrete chaos process. Related to the aforementioned work, the author in \cite{lech}  develops the theory related to the  two-parameter tempered Hermite field, including moving
average, sample path properties, spectral representations and the theory of Wiener stochastic
integration with respect to the two-parameter tempered Hermite field of order one. In \cite{bon} the authors  define two new classes of stochastic processes, called tempered fractional Lévy processes of the first and second kinds. These processes, as usual, are constructed by  exponentially tempering the power law kernel in the moving average representation of a fractional Lévy process. Using the framework of tempered fractional integrals and derivates the authors develop the theory of stochastic integration with respect to both processes.  In this same direction we can mention \cite{dilip, fan, mer}.

The remainder of this paper is organized as follows. In Section 2, we briefly recall some relevant aspects and notations related to generalized kernels. Section 3 is devoted to the construction of the process and the study of the properties of stationary increments, scaling, regularity of the paths and, under some special assumptions, its covariance structure. In Section 4, a filtered version of the same process is constructed, and as before, the same properties are established. Section 5, is concerned with the analysis of the process in the case of the Hermite kernel (regular kernel and filtered). In Section 6, we consider the problem of non parametric estimation. Precisely, we consider a co-integrated regressor model where the regressor is a fractional Brownian motion with Hurst parameter $H_{1} \in (0, 1)$ and $Z^{\lambda}$ is a generalized tempered Hermite process. Finally, Section 7 contains brief appendix related to elementary topics of Malliavin calculus.

\section{Preliminaries}\label{Pre}
We briefly recall some relevant aspects and notations related to generalized kernels, our main reference is \cite{bai}.

\subsection{Generalized kernels}
In this part,  we introduce the definition and properties of a general type of kernel that we will use in order to construct our tempered process. 

We need to introduce the following notation, let ${\bf z} = (z_{1}, \ldots , z_{k}) \in \mathbb{R}^{k}$, ${\bf i} = (i_{1}, \ldots , i_{k}) $, ${\bf 0} = (0, \ldots , 0) $; ${\bf 1} = (1, \ldots , 1) $, where the dimensions of ${\bf 0}$ and ${\bf 1}$ will be determined  by the context of the computations. Also, for any real number $z$, $[z] = \sup \{ n \in \mathbb{Z} , n \leq z  \}$, and  $[{\bf z}] = ([z_{1}], \ldots , [z_{k}]) $. Furthermore, we write ${\bf x} > {\bf y}$ for $x_{j} > y_{j}$ with $j=1 \ldots, k$ (the defnition for $\geq$ is exactly the same). $< {\bf y} , {\bf z}> = \sum_{j=1}^{k} y_{j}z_{j}$, $\Vert {\bf z} \Vert = < {\bf z} , {\bf z}> $. We use $\Vert \cdot \Vert$ with a subscript to denote the norm of some other space (the space will be determined by the subscript).  For a set $B \subset \mathbb{R}$, $B^{k}$ is the $k$-fold cartesian product. To simplify the computations we establish the following notation for the tempering factor 
$$  e^{-\lambda (s{\bf 1} - {\bf x})} := \prod_{i=1}^{k} e^{-\lambda (s-x_{i})} .$$  
Now, we recall a proposition from the reference \cite{bai}. This proposition allow us to construct general H-ssi process living in a Wiener chaos
\begin{prop}\label{def-1}
Let $H \in (0,1)$. Suppose that $\{ h_{t}(\cdot) , t >0 \}$ is a family of functions defined on $\mathbb{R}^{k}$ satisfying
\begin{enumerate}
\item $h_{t} \in L^{2}(\mathbb{R}^{k})$;
\item $\forall \lambda >0 ,  \exists \beta \neq 0$, such that $h_{\lambda t } ({\bf x}) = \lambda^{H + k\beta /2  } h_{t}(\lambda^{\beta} {\bf x})$ for a.e. ${\bf x} \in \mathbb{R}^{k}$ and all $t>0$;
\item $\forall s, \exists $ {\bf a} $\in \mathbb{R}^{k}$, such that $h_{t+s}({\bf x}) - h_{t}({\bf x}) = h_{s}({\bf x} + t {\bf a})$ for a.e. ${\bf x} \in \mathbb{R}^{k}$ and all $t >0$.  
\end{enumerate}  
Then $$Z(t):=I_{k}(h_{t})=\int_{\mathbb{R}^{k}}^{'} \int_{0}^{t}  g(s-x_{1}, \ldots , s-x_{k})  1_{\{ s>x_{1}, \ldots , s>x_{k} \}} ds  W(dx_{1}) \ldots W(dx_{k}) $$
is an H-sssi process with
$$h_{t}({\bf x}) := \int_{0}^{t} g(s{\bf 1} - {\bf x}) 1_{\{  s{\bf 1} > {\bf x} \}}ds.$$  
\end{prop}
Now,  the idea is to give the conditions over a general function $g$, such that, we can verify the conditions of Proposition \ref{def-1} 

\begin{definition}\label{hermite-kernel}
A nonzero measurable function $g$ defined on $\mathbb{R}^{k}_{+}$ is called {\it generalized tempered Hermite kernel}, if it satisfies the following conditions:
\begin{itemize}
\item[(H1)] $g(c {\bf x}) = c^{\alpha} g({\bf x})$,  $\forall c >0$, where $\alpha \in \left(-\dfrac{k+1}{2}, -\dfrac{k}{2} \right)$.
\item[(H2)] $\int_{\mathbb{R}^{k}_{+}} d{\bf x} \vert g({\bf x})  g({\bf 1} + {\bf x}) \vert e^{-2\lambda u {\bf x}} < \infty $.
\end{itemize}
\end{definition}

 
\section{The Process}
In this part, we introduce the tempered  generalized Hermite  process. Then, we study some sample properties of the process: scaling, stationarity of the increments, covariance and sample path regularity. \\
 Let $\lambda > 0$, the tempered generalized Hermite process is given by 
\begin{eqnarray}
Z^{\lambda}(t) &:=& \int_{\mathbb{R}^{k}}^{'}  \int_{0}^{t} g(s-x_{1}, \ldots , s-x_{k}) \prod_{i=1}^{k} e^{-\lambda (s-x_{i})} 1_{\{ s>x_{1}, \ldots , s>x_{k} \}} ds  W(dx_{1}) \ldots W(dx_{k})\nonumber \\
&=& I_{k} (h_{t}^{\lambda}),  \label{process}
\end{eqnarray}
where 
\begin{equation}\label{tempering-0}
h_{t}^{\lambda}({\bf x}) := \int_{0}^{t} g(s{\bf 1} - {\bf x}) e^{-\lambda (s{\bf 1} - {\bf x})}  1_{\{  s{\bf 1} > {\bf x} \}}ds,
\end{equation}
in here, to simplify the computations we establish the following notation for the tempering factor 
\begin{equation}\label{tempering} 
 e^{-\lambda (s{\bf 1} - {\bf x})} := \prod_{i=1}^{k} e^{-\lambda (s-x_{i})} .
\end{equation}

Next, we will prove that $Z^{\lambda}$ is well defined. In fact, we have the following result: 
\begin{prop}\label{finite-1}
Let $g({\bf x})$ be a generalized Hermite kernel presented in Definition \ref{hermite-kernel}. Then,
$$ h_{t}^{\lambda}({\bf x}) := \int_{0}^{t} g(s{\bf 1} - {\bf x}) e^{-\lambda (s{\bf 1} - {\bf x})}  1_{\{  s{\bf 1} > {\bf x} \}} ds $$    
is well defined in $L^{2}(\mathbb{R}^{k}) $.
\end{prop}
\begin{proof}
To check that $h_{t}^{\lambda}$ is well defined, we write 
\begin{eqnarray*}
\int_{\mathbb{R}^{k}} {h_{t}^{\lambda}({\bf x})}^{2} d{\bf x} &=& \int_{\mathbb{R}^{k}} d{\bf x} \int_{0}^{t} \int_{0}^{t} ds_{1}ds_{2}  g(s_{1}{\bf 1} - {\bf x})  g(s_{2}{\bf 1} - {\bf x}) \\
& \times &  e^{-\lambda (s_{1}{\bf 1} - {\bf x})}  e^{-\lambda (s_{2}{\bf 1} - {\bf x})} 1_{\{  s_{1}{\bf 1} > {\bf x} \}} 1_{\{  s_{2}{\bf 1} > {\bf x} \}},
\end{eqnarray*}
now, we want to use Fubini theorem. Therefore, we need to check that the absolute value of the integrand is integrable,  i.e. 
\begin{eqnarray*}
\int_{\mathbb{R}^{k}} {h_{t}^{\lambda}({\bf x})}^{2} d{\bf x}&=&  2 \int_{0}^{t} ds_{1} \int_{s_{1}}^{t} ds_{2} \int_{\mathbb{R}^{k}} d{\bf x} \vert g(s_{1}{\bf 1} - {\bf x})  g(s_{2}{\bf 1} - {\bf x})\vert  \\
& \times &  e^{-\lambda (s_{1}{\bf 1} - {\bf x})}  e^{-\lambda (s_{2}{\bf 1} - {\bf x})} 1_{\{  s_{1}{\bf 1} > {\bf x} \}} 1_{\{  s_{2}{\bf 1} > {\bf x} \}}, \\
\end{eqnarray*}
also, we have  used the symmetry  in  $s_{1}$ and $s_{2}$. Now, making the change of variables $s=s_{1}$, $u = s_{2}-s_{1}$, $w = s_{1} {\bf 1} - {\bf x}$, we get
  \begin{eqnarray*}
\int_{\mathbb{R}^{k}} {h_{t}^{\lambda}({\bf x})}^{2} d{\bf x}&=& 2 \int_{0}^{t} ds \int_{0}^{t-s} du \int_{\mathbb{R}^{k}_{+}} d{\bf w} \vert  g({\bf w})  g(u{\bf 1} + {\bf w})\vert e^{-\lambda ( {\bf w})}  e^{-\lambda (u{\bf 1} + {\bf w})} \\
&=&  2 \int_{0}^{t} ds \int_{0}^{t-s} du u^{k+2\alpha}  e^{-\lambda u k} \int_{\mathbb{R}^{k}_{+}} d{\bf y} \vert  g({\bf y}) g({\bf 1 + y}) \vert e^{-2\lambda u {\bf y}},
\end{eqnarray*} 
here, in the last equality, we have used the change of variables ${\bf y} u = {\bf w}$, the condition (H1), and condition (H2). We can see that the last expression is finite by $2\alpha + k +1 >0$ and (H2). If we compute the same without the absolute value we arrive to 
 \begin{eqnarray*}
\int_{\mathbb{R}^{k}} {h_{t}^{\lambda}({\bf x})}^{2} d{\bf x} &=&  2 \int_{0}^{t} ds \int_{0}^{t-s} du u^{k+2\alpha}  e^{-\lambda u k} \int_{\mathbb{R}^{k}_{+}} d{\bf y}   g({\bf y}) g({\bf 1 + y})  e^{-2\lambda u {\bf y}}\\
& \leq & \dfrac{t^{2\alpha + k + 2}}{(\alpha + k/2 + 1)(2\alpha + k + 2)}  \int_{\mathbb{R}^{k}_{+}} d{\bf y}  g({\bf y}) g({\bf 1 + y}) e^{-2\lambda u {\bf y}}
\end{eqnarray*} 
wich is finite for the same reason as before.
\end{proof}

\begin{remark}
If we consider necessary, in some of the proofs, to be more clear in the computations, we will use the product notation, i.e., we will not use the notation defined in \eqref{tempering-0} and \eqref{tempering}.  
\end{remark}

\begin{lemma}\label{sssi-1}
Let $H>1/2$ and $\lambda > 0$. Then,  the tempered generalized  Hermite process $Z^{\lambda}$   is a stationary increments process with the scaling property 
$$ \lbrace Z^{\lambda}(ct)  \rbrace_{t \in \mathbb{R}} \overset{d}{=}  \lbrace c^{H} Z^{c \lambda }(t)  \rbrace_{t \in \mathbb{R}} ,$$
\end{lemma}
where $c>0$ and $\overset{d}{=}$ means equality in sense of finite dimensional distributions. 
\begin{proof}
To check the stationarity of the increments, we write for $t,h>0$

\begin{eqnarray*}
Z^{\lambda}(t + h) - Z^{\lambda}(t) &=&  \int_{\mathbb{R}^{k}}^{'} \int_{0}^{t+h}  g(s-x_{1}, \ldots , s-x_{k}) \prod_{i=1}^{k} e^{-\lambda (s-x_{i})} 1_{\{ s>x_{1}, \ldots , s>x_{k} \}} ds  W(dx_{1}) \ldots W(dx_{k}) \\
&-&  \int_{\mathbb{R}^{k}}^{'}  \int_{0}^{t}  g(s-x_{1}, \ldots , s-x_{k}) \prod_{i=1}^{k} e^{-\lambda (s-x_{i})} 1_{\{ s>x_{1}, \ldots , s>x_{k} \}} ds  W(dx_{1}) \ldots W(dx_{k}) \\
&=&  \int_{\mathbb{R}^{k}}^{'} \int_{t}^{t+h} g(s-x_{1}, \ldots , s-x_{k}) \prod_{i=1}^{k} e^{-\lambda (s-x_{i})} 1_{\{ s>x_{1}, \ldots , s>x_{k} \}} ds  W(dx_{1}) \ldots W(dx_{k}). 
\end{eqnarray*}
Making the change of variable $u=s-t$, we obtain
\begin{eqnarray*}
Z^{\lambda}(t + h) - Z^{\lambda}(t) &=& \int_{\mathbb{R}^{k}}^{'} \int_{0}^{h} g(u+t-x_{1}, \ldots , u+t-x_{k}) \\
&\times &\prod_{i=1}^{k} e^{-\lambda (u+t-x_{i})} 1_{\{ u+t>x_{1}, \ldots , u+t>x_{k} \}} du  W(dx_{1}) \ldots W(dx_{k}). 
\end{eqnarray*}
Now, by the change of variable $y_{i} = x_{i} +t$, we get 
 \begin{eqnarray*}
Z^{\lambda}(t + h) - Z^{\lambda}(t) &=& \int_{\mathbb{R}^{k}}^{'} \int_{0}^{h} g(u-y_{1}, \ldots , u-y_{k}) \\
&\times &\prod_{i=1}^{k} e^{-\lambda (u-y_{i})} 1_{\{ u>y_{1}, \ldots , u>y_{k} \}} du  W(d(y_{1}+t)) \ldots W(d(y_{k}+t))\\
&\overset{d}{=}& Z^{\lambda}(h),
\end{eqnarray*}
where the last equality comes from the stationarity of Brownian motion. Now, to prove the scaling property, we  have for $c>0$
\begin{eqnarray*}
Z^{\lambda}(ct) &=&  \int_{\mathbb{R}^{k}}^{'}  \int_{0}^{ct}  g(s-x_{1}, \ldots , s-x_{k}) \prod_{i=1}^{k} e^{-\lambda (s-x_{i})} 1_{\{ s>x_{1}, \ldots , s>x_{k} \}} ds  W(dx_{1}) \ldots W(dx_{k}) \\
&=& c  \int_{\mathbb{R}^{k}}^{'} \int_{0}^{t} g(uc-x_{1}, \ldots , uc-x_{k}) \prod_{i=1}^{k} e^{-\lambda (uc-x_{i})} 1_{\{ uc>x_{1}, \ldots , uc>x_{k} \}} du  W(dx_{1}) \ldots W(dx_{k}), 
\end{eqnarray*}
here, we used the change of variable $u=s/c$. Continuing, we use the change of variable $y_{i} = x_{i}/c$ 
\begin{eqnarray*}
Z^{\lambda}(ct) &=& c \int_{\mathbb{R}^{k}}^{'} \int_{0}^{t} g(uc- cy_{1}, \ldots , uc-cy_{k}) \prod_{i=1}^{k} e^{-\lambda (uc-cy_{i})} 1_{\{ uc>cy_{1}, \ldots , uc>cy_{k} \}} du  W(d(cy_{1})) \ldots W(d(cy_{k}))\\
&\overset{d}{=}& c^{\alpha +1 + k/2} Z^{c\lambda}(t),
\end{eqnarray*}
where, in the last equality we have used the condition $(H1)$ and the self-similarity of Brownian motion. Finally,  the result is achieved by taking $\alpha +1 = H-k/2$. 

\end{proof}

\begin{lemma}\label{regularidad}
The stochastic process $Z^{\lambda}$ has a continuos version.
\end{lemma}
\begin{proof}
By Proposition \ref{finite-1} and Lemma \ref{sssi-1}, we can get
$$ E(\vert Z^{\lambda}(t) - Z^{\lambda}(s) \vert^{2}) \leq C \vert t-s\vert^{2\alpha + k +2},$$
noticing that taking $\alpha +1 = H + k/2$, we have 
$$ E(\vert Z^{\lambda}(t) - Z^{\lambda}(s) \vert^{2}) \leq C \vert t-s\vert^{2H},$$
where $H \in (1/2,1)$. Then, the result is achieved by means of Kolmogorov Chensov theorem. 
\end{proof}

\begin{lemma}
Let us assume that $g$ given by  Definition \ref{hermite-kernel} with
\begin{equation}\label{prop-g}
g(x_{1}, \ldots , x_{k}) = \prod_{i=1}^{k} g(x_{i})
\end{equation}
and  $g(0)=0$ for $i=1, \ldots, k$. Then,  the tempered generalized  Hermite process $Z^{\lambda}$   has the covariance function 
$$ E(Z^{\lambda}(t) Z^{\lambda}(s) ) = k!  \int_{0}^{t} \int_{0}^{s}   e^{-\lambda k \vert u-v \vert} \vert u-v \vert^{k(2 \alpha -1)}  \left[ \int_{0}^{\infty} g[x] g[1+x] e^{-2\lambda \vert u-v \vert x}    dx \right]^{k}   du dv,$$
where $\lambda > 0$, and $\alpha \in \left(-\dfrac{k+1}{2}, -\dfrac{k}{2} \right)$ (equivalently $H>1/2$). 
\end{lemma}
\begin{proof}
By the defintion of $Z^{\lambda}$, the fact that $g(0)=0$, Fubini theorem and the isometry of multiple Wiener - It\^o integrals, we get
\begin{eqnarray*}
E(Z^{\lambda}(t) Z^{\lambda}(s) ) &=& E[I_{k}(h_{t}^{\lambda}) I_{k}(h_{s}^{\lambda})] \\
&=& k! \int_{\mathbb{R}^{k}} {h_{t}^{\lambda}({\bf x})} {h_{s}^{\lambda}({\bf x})} d{\bf x}\\
&=&  k!  \int_{0}^{t} \int_{0}^{s}  \left[ \int_{\mathbb{R}^{k}}   g[(u{\bf 1} - {\bf x})_{+}]  g[(v{\bf 1} - {\bf x})_{+}]   e^{-\lambda (u{\bf 1} - {\bf x})_{+}}  e^{-\lambda (v{\bf 1} - {\bf x})_{+}} d{\bf x} \right]   du dv,
\end{eqnarray*}
Now, using \eqref{prop-g} we can obtain 
\begin{eqnarray*}
E(Z^{\lambda}(t) Z^{\lambda}(s) ) &=&  k!  \int_{0}^{t} \int_{0}^{s}  \left[ \int_{\mathbb{R}^{k}}  \prod_{i=1}^{k} g[(u-x_{i})_{+}] g[(v-x_{i})_{+}] e^{-\lambda (u-x_{i})_{+}}  e^{-\lambda (v-x_{i})_{+}}  dx_{1} \cdots  dx_{k} \right]   du dv \\
&=&  k!  \int_{0}^{t} \int_{0}^{s}  \left[ \int_{\mathbb{R}} g[(u-x)_{+}] g[(v-x)_{+}] e^{-\lambda (u-x)_{+}}  e^{-\lambda (v-x)_{+}}  dx \right]^{k}   du dv \\
&=& k!  \int_{0}^{t} \int_{0}^{s}  \left[ \int_{-\infty}^{u \wedge v} g[u-x] g[v-x] e^{-\lambda (u-x)}  e^{-\lambda (v-x)}  dx \right]^{k}   du dv \\
&=& k!  \int_{0}^{t} \int_{0}^{s}   e^{-\lambda k \vert u-v \vert}  \left[ \int_{0}^{\infty} g[w] g[\vert u-v\vert + w] e^{-2\lambda w}    dw \right]^{k}   du dv.
\end{eqnarray*}
By the properties of $g$
\begin{eqnarray*}
E(Z^{\lambda}(t) Z^{\lambda}(s) ) &=&  k!  \int_{0}^{t} \int_{0}^{s}   e^{-\lambda k \vert u-v \vert} \vert u-v \vert^{k(2 \alpha -1)}  \left[ \int_{0}^{\infty} g[x] g[1+x] e^{-2\lambda \vert u-v \vert x}    dx \right]^{k}   du dv.
\end{eqnarray*}
\end{proof}

\section{The process: Fractionally filtered kernels}\label{process-filtered}
In this part, we introduce the tempered  generalized Hermite  process. Then, we study some sample properties of the process: scaling, stationarity of the increments, covariance and sample path regularity. \\
Let  $\lambda > 0$, we define the tempered generalized Hermite  process with filtered kernel  by 
\begin{eqnarray}
Z^{\lambda, \beta}(t) &:=& \int_{\mathbb{R}^{k}}^{'} \int_{\mathbb{R}} g(s-x_{1}, \ldots , s-x_{k}) \dfrac{1}{\beta} [(t-s)_{+}^{\beta} - (-s)_{+}^{\beta}]  \nonumber \\
&\times& \prod_{i=1}^{k} e^{-\lambda (s-x_{i})} 1_{\{ s>x_{1}, \ldots , s>x_{k} \}} ds  W(dx_{1}) \ldots W(dx_{k}) \nonumber \\
&=& I_{k} (h_{t}^{\lambda, \beta}), \label{filter-hermite}
\end{eqnarray}
where 
$$h_{t}^{\lambda} := \int_{\mathbb{R}} l_{t}(s)^{\beta} g(s{\bf 1} - {\bf x}) e^{-\lambda (s{\bf 1} - {\bf x})}  1_{\{  s{\bf 1} > {\bf x} \}}ds,$$
with
$$ l_{t}^{\beta}(s)  =  \dfrac{1}{\beta} [(t-s)_{+}^{\beta} - (-s)_{+}^{\beta}]$$ 
and $  e^{-\lambda (s{\bf 1} - {\bf x})}$ given by (\ref{tempering}).

Next, we will prove that $Z^{\lambda , \beta}$ is well defined. In fact, we have the following result: 

\begin{prop}\label{finite-2}
Let $g({\bf x})$ be a generalized Hermite kernel given in Definition \ref{hermite-kernel}. If
$$-1 < -\alpha - \dfrac{k}{2} -1 < \beta < - \alpha - \dfrac{k}{2} < \dfrac{1}{2}, \quad \beta \neq 0.$$
Then,
$$ h_{t}^{\lambda, \beta} := \int_{\mathbb{R}}   l_{t}^{\beta}(s)  g(s{\bf 1} - {\bf x}) e^{-\lambda (s{\bf 1} - {\bf x})}  1_{\{  s{\bf 1} > {\bf x} \}} ds $$    
is well defined in $L^{2}(\mathbb{R}^{k}) $.
\end{prop}
\begin{proof}
By the defintion of $ h^{\lambda, \beta} $, we get
\begin{eqnarray*}
\int_{R^{k}} {h_{t}^{\lambda, \beta}({\bf x})}^{2} d{\bf x} & \leq & 2 \int_{\mathbb{R}} ds_{1} \int_{s_{1}}^{\infty} ds_{2} \int_{\mathbb{R}^{k}} d{\bf x}  l_{t}^{\beta}(s_{1}) l_{t}^{\beta}(s_{2}) \\
& \times & \vert g(s_{1}{\bf 1} - {\bf x})  g(s_{2}{\bf 1} - {\bf x})\vert  e^{-\lambda (s_{1}{\bf 1} - {\bf x})}  e^{-\lambda (s_{2}{\bf 1} - {\bf x})} 1_{\{  s_{1}{\bf 1} > {\bf x} \}}. 
\end{eqnarray*}
Making the change of variables $s=s_{1}$, $u = s_{2}-s_{1}$, $w = s_{1} {\bf 1} - {\bf x}$, we get
\begin{eqnarray*}
\int_{R^{k}} {h_{t}^{\lambda, \beta}({\bf x})}^{2} d{\bf x} & \leq & 2 \int_{\mathbb{R}} ds \int_{0}^{\infty} du  l_{t}^{\beta}(s) l_{t}^{\beta}(s + u) \int_{\mathbb{R}_{+}^{k}} d{\bf x}  \\
& \times &  \vert  g({\bf w})  g(u{\bf 1} + {\bf w})\vert e^{-\lambda ( {\bf w})}  e^{-\lambda (u{\bf 1} + {\bf w})} \\ 
& \leq & 2 \int_{\mathbb{R}} ds  l_{t}^{\beta}(s)  \int_{0}^{\infty} du l_{t}^{\beta}(s + u)u^{2\alpha +k} e^{-\lambda uk}\\
& \times & \int_{\mathbb{R}_{+}^{k}} d{\bf y}  \vert  g({\bf y})  g({\bf 1} + {\bf y})\vert   e^{- 2\lambda u {\bf y}}.
\end{eqnarray*}
By condition $(H2)$ and the fact $\lambda >0$, we need to prove that 
$$ \int_{\mathbb{R}} ds  l_{t}^{\beta}(s)  \int_{0}^{\infty} du l_{t}^{\beta}(s + u)u^{2\alpha +k} e^{-\lambda uk} < \infty.$$
However, since $u \in (0, \infty)$, we have 
\begin{eqnarray*}
 \int_{\mathbb{R}} ds  l_{t}^{\beta}(s)  \int_{0}^{\infty} du l_{t}^{\beta}(s + u)u^{2\alpha +k} e^{-\lambda uk} 
& <& \int_{\mathbb{R}} ds  l_{t}^{\beta}(s)  \int_{0}^{\infty} du l_{t}^{\beta}(s + u)u^{2\alpha +k} 
\end{eqnarray*}
and this last term if finite for
$$-1 < -\alpha - \dfrac{k}{2} -1 < \beta < - \alpha - \dfrac{k}{2} < \dfrac{1}{2}, \quad \beta \neq 0$$
due to Proposition 3.25 in  \cite{bai}.  
\end{proof}

\begin{lemma}\label{sssi-2}
Let $\lambda > 0$, and 
$$-1 < -\alpha - \dfrac{k}{2} -1 < \beta < - \alpha - \dfrac{k}{2} < \dfrac{1}{2}, \quad \mbox{with} \ \beta \neq 0.$$
Then,  the tempered generalized  Hermite process $Z^{\lambda, }$   is a stationary increments process with the scaling property 
$$ \lbrace Z^{\lambda, \beta}(ct)  \rbrace_{t \in \mathbb{R}} \overset{d}{=}  \lbrace c^{\beta + 1 + \alpha + k/2}  Z^{ c \lambda, \beta}(t)  \rbrace_{t \in \mathbb{R}} ,$$
\end{lemma}
where $c>0$ and $\overset{d}{=}$ means equality in sense of finite dimensional distributions. 
\begin{proof}
To check the stationarity of the increments, we write for $t, h>0$ 
\begin{eqnarray*}
Z^{\lambda, \beta}(t + h) - Z^{\lambda, \beta}(t) &=&   \int_{\mathbb{R}}^{'} \dfrac{1}{\beta} [(t+h-s)_{+}^{\beta} - (-s)_{+}^{\beta}] g(s-x_{1}, \ldots , s-x_{k}) \\
&\times& \prod_{i=1}^{k} e^{-\lambda (s-x_{i})} 1_{\{ s>x_{1}, \ldots , s>x_{k} \}} ds  W(dx_{1}) \ldots W(dx_{k}) \\
&-&   \int_{\mathbb{R}}^{'} \dfrac{1}{\beta} [(t-s)_{+}^{\beta} - (-s)_{+}^{\beta}] g(s-x_{1}, \ldots , s-x_{k}) \\
&\times& \prod_{i=1}^{k} e^{-\lambda (s-x_{i})} 1_{\{ s>x_{1}, \ldots , s>x_{k} \}} ds  W(dx_{1}) \ldots W(dx_{k}) \\
&=&\int_{\mathbb{R}}^{'} \dfrac{1}{\beta} [(t+h-s)_{+}^{\beta} - (t-s)_{+}^{\beta}] g(s-x_{1}, \ldots , s-x_{k}) \\
&\times& \prod_{i=1}^{k} e^{-\lambda (s-x_{i})} 1_{\{ s>x_{1}, \ldots , s>x_{k} \}} ds  W(dx_{1}) \ldots W(dx_{k}). 
\end{eqnarray*}
Making the change of variable $t-s = -v$, we get
\begin{eqnarray*}
Z^{\lambda, \beta}(t + h) - Z^{\lambda, \beta}(t) &=&\int_{\mathbb{R}}^{'} \dfrac{1}{\beta} [(h-v)_{+}^{\beta} - (-v)_{+}^{\beta}] g(t+v-x_{1}, \ldots , t+v-x_{k}) \\
&\times& \prod_{i=1}^{k} e^{-\lambda (t+v-x_{i})} 1_{\{ t+v>x_{1}, \ldots , t+v>x_{k} \}} dv  W(dx_{1}) \ldots W(dx_{k}). 
\end{eqnarray*}
Now, by the change of variable $y_{i} = x_{i}- t $ for $i=1, \ldots , k$, we obtain
\begin{eqnarray*}
Z^{\lambda, \beta}(t + h) - Z^{\lambda, \beta}(t) &=&\int_{\mathbb{R}}^{'} \dfrac{1}{\beta} [(h-v)_{+}^{\beta} - (-v)_{+}^{\beta}] g(v-y_{1}, \ldots , v-y_{k}) \\
&\times& \prod_{i=1}^{k} e^{-\lambda (v-y_{i})} 1_{\{ v>y_{1}, \ldots , v>y_{k} \}} du  W(d(y_{1} +t)) \ldots W(d(y_{k} + t)) \\
&\overset{d}{=}&  Z^{\lambda, \beta}(h).
\end{eqnarray*}
This last equality is due to the stationarity of Brownian motion.  With respecto to the scaling property, we have for $c >0$
\begin{eqnarray*}
Z^{\lambda, \beta}(ct) &=&   \int_{\mathbb{R}}^{'} \dfrac{1}{\beta} [(ct-s)_{+}^{\beta} - (-s)_{+}^{\beta}] g(s-x_{1}, \ldots , s-x_{k}) \\
&\times& \prod_{i=1}^{k} e^{-\lambda (s-x_{i})} 1_{\{ s>x_{1}, \ldots , s>x_{k} \}} ds  W(dx_{1}) \ldots W(dx_{k}), 
\end{eqnarray*}
by the change of variable $v = s/c$, we get  
\begin{eqnarray*}
Z^{\lambda, \beta}(ct) &=&  c^{\beta + 1}  \int_{\mathbb{R}}^{'} \dfrac{1}{\beta} [(t-v)_{+}^{\beta} - (-v)_{+}^{\beta}] g(cv-x_{1}, \ldots , cv-x_{k}) \\
&\times& \prod_{i=1}^{k} e^{-\lambda (cv-x_{i})} 1_{\{ cv>x_{1}, \ldots , cv>x_{k} \}} dv  W(dx_{1}) \ldots W(dx_{k}), 
\end{eqnarray*}
in a similar way, we use the change of variable $y_{i} = x_{i}/c$, this allow us to obtain 
 \begin{eqnarray*}
Z^{\lambda, \beta}(ct) &=&  c^{\beta + 1 + \alpha}  \int_{\mathbb{R}}^{'} \dfrac{1}{\beta} [(t-v)_{+}^{\beta} - (-v)_{+}^{\beta}] g(v-y_{1}, \ldots , v-y_{k}) \\
&\times& \prod_{i=1}^{k} e^{- c \lambda (v-x_{i})} 1_{\{ v>y_{1}, \ldots , v>y_{k} \}} dv  W(dcy_{1}) \ldots W(dcy_{k})\\ 
&\overset{d}{=}&  c^{\beta + 1 + \alpha + k/2}  \int_{\mathbb{R}}^{'} \dfrac{1}{\beta} [(t-v)_{+}^{\beta} - (-v)_{+}^{\beta}] g(v-y_{1}, \ldots , v-y_{k}) \\
&\times& \prod_{i=1}^{k} e^{- c \lambda (v-x_{i})} 1_{\{ v>y_{1}, \ldots , v>y_{k} \}} dv  W(dy_{1}) \ldots W(dy_{k})\\
&=&  c^{\beta + 1 + \alpha + k/2}  Z^{ c \lambda, \beta}(t).
\end{eqnarray*}
\end{proof}
\begin{remark}
We can take $H =\beta + 1 + \alpha + k/2  \in (0,1)$, then 
$$ \lbrace Z^{\lambda, \beta}(ct)  \rbrace_{t \in \mathbb{R}} \overset{d}{=}  \lbrace c^{H}  Z^{ c \lambda, \beta}(t)  \rbrace_{t \in \mathbb{R}} ,$$
if we want to consider the anti-persistent case $H<1/2$, then we have to take 
$$\beta \in \left( - \alpha - \dfrac{k}{2} -1 , -\alpha - \dfrac{k}{2} - \dfrac{1}{2}\right).$$
\end{remark}
\begin{lemma}\label{cov-fill}
Let us assume that $g$ given by  Definition \ref{hermite-kernel} with
\begin{equation}\label{prop-g}
g(x_{1}, \ldots , x_{k}) = \prod_{i=1}^{k} g(x_{i})
\end{equation}
and  $g(0)=0$ for $i=1, \ldots, k$. Then,  the tempered generalized  Hermite process $Z^{\lambda}$   has the covariance function 
$$ E(Z^{\lambda}(t) Z^{\lambda}(s) ) = k!  \int_{0}^{t} \int_{0}^{s}   e^{-\lambda k \vert u-v \vert} \vert u-v \vert^{k(2 \alpha -1)}  l^{\beta}_{s}(u) l^{\beta}_{t}(v)  \left[ \int_{0}^{\infty} g[x] g[1+x] e^{-2\lambda \vert u-v \vert x}    dx \right]^{k}   du dv.$$
\end{lemma}
\begin{proof}
By the defintion of $Z^{\lambda}$, the fact that $g(0)=0$, Fubini theorem and the isometry of multiple Wiener - It\^o integrals, we get
\begin{eqnarray*}
E(Z^{\lambda}(t) Z^{\lambda}(s) ) &=& E[I_{k}(h_{t}^{\lambda}) I_{k}(h_{s}^{\lambda})] \\
&=& k! \int_{\mathbb{R}^{k}} {h_{t}^{\lambda}({\bf x})} {h_{s}^{\lambda}({\bf x})} d{\bf x}\\
&=&  k!  \int_{\mathbb{R}} \int_{\mathbb{R}}  l^{\beta}_{s}(u) l^{\beta}_{t}(v) \left[ \int_{\mathbb{R}^{k}}   g[(u{\bf 1} - {\bf x})_{+}]  g[(v{\bf 1} - {\bf x})_{+}]   e^{-\lambda (u{\bf 1} - {\bf x})_{+}}  e^{-\lambda (v{\bf 1} - {\bf x})_{+}} d{\bf x} \right]   du dv,
\end{eqnarray*}
Now, using \eqref{prop-g} we can obtain 
\begin{eqnarray*}
E(Z^{\lambda}(t) Z^{\lambda}(s) ) &=&  k!  \int_{\mathbb{R}} \int_{\mathbb{R}}  l^{\beta}_{s}(u) l^{\beta}_{t}(v)  \left[ \int_{\mathbb{R}^{k}}  \prod_{i=1}^{k} g[(u-x_{i})_{+}] g[(v-x_{i})_{+}] e^{-\lambda (u-x_{i})_{+}}  e^{-\lambda (v-x_{i})_{+}}  dx_{1} \cdots  dx_{k} \right]   du dv \\
&=&  k!  \int_{\mathbb{R}} \int_{\mathbb{R}}  l^{\beta}_{s}(u) l^{\beta}_{t}(v)  \left[ \int_{\mathbb{R}} g[(u-x)_{+}] g[(v-x)_{+}] e^{-\lambda (u-x)_{+}}  e^{-\lambda (v-x)_{+}}  dx_{1} \cdots  dx_{k} \right]^{k}   du dv \\
&=& k!  \int_{\mathbb{R}} \int_{\mathbb{R}}  l^{\beta}_{s}(u) l^{\beta}_{t}(v)  \left[ \int_{-\infty}^{u \wedge v} g[u-x] g[v-x] e^{-\lambda (u-x)}  e^{-\lambda (v-x)}  dx \right]^{k}   du dv \\
&=& k!  \int_{\mathbb{R}} \int_{\mathbb{R}}   l^{\beta}_{s}(u) l^{\beta}_{t}(v)  e^{-\lambda k \vert u-v \vert}  \left[ \int_{0}^{\infty} g[w] g[\vert u-v\vert + w] e^{-2\lambda w}    dw \right]^{k}   du dv.
\end{eqnarray*}
By the properties of $g$
\begin{eqnarray*}
E(Z^{\lambda}(t) Z^{\lambda}(s) ) &=&  k!  \int_{\mathbb{R}} \int_{\mathbb{R}}  e^{-\lambda k \vert u-v \vert} \vert u-v \vert^{k(2 \alpha -1)}  l^{\beta}_{s}(u) l^{\beta}_{t}(v)  \left[ \int_{0}^{\infty} g[x] g[1+x] e^{-2\lambda \vert u-v \vert x}    dx \right]^{k}   du dv.
\end{eqnarray*}
\end{proof}

\newpage
\section{The Hermite case}

Here, we consider and study  a  special case of kernel that fulfill the conditions of Proposition \ref{def-1} and Definition \ref{hermite-kernel}.  We will consider the Hermite kernel and  then,  the filtered version of the same kernel. 

\subsection{Hermite kernel}
Let us recall that the Hermite kernel is given by 
$$g({\bf x}) = \prod_{j=1}^{k} x_{j}^{d-1}, \quad x_{j}>0.$$
Clearly, this kernel  meets all the conditions of Definition \ref{hermite-kernel}. In fact, in \cite{sab} the autor  study the following process
\begin{equation}\label{hermite-1}
Z^{\lambda}(t) = \int_{\mathbb{R}^{k}}^{'} \int_{0}^{t} \prod_{i=1}^{k} (s-x_{i})^{d-1}_{+} e^{-\lambda (s-x_{i})_{+}}ds  W(dx_{1}) \ldots W(dx_{k})  ,  
\end{equation}
where $\lambda >0$, $(x)_{+} = x1_{\{ x>0 \}}$ and $d=\dfrac{1}{2}- \dfrac{1-H}{k} \in \left( \dfrac{1}{2} - \dfrac{1}{2k} , \infty \right)$, and $H>1/2$. By Proposition \ref{finite-1} and Lemma \ref{sssi-1}  we know that $Z^{\lambda}$ is a stationary increment process with a scaling property. Furthermore, the autor in \cite{sab} obtains the following properties given in Proposition \ref{cov-hermite} and Proposition \ref{spec-hermite}  (for all the proofs the reader can refer to \cite{sab}).

\begin{prop}\label{cov-hermite}
Let $Z^{\lambda}$ be the process given by (\ref{hermite-1}) has the covariance function 
$$E(Z^{\lambda}(t)Z^{\lambda}(s)) = 2 \left[ \dfrac{\Gamma (d)}{\sqrt{\pi} (2\lambda)^{d-1/2}} \right]^{k} \int_{0}^{t}\int_{0}^{s} \left[ \vert u-v\vert^{d-1/2} K_{1/2-d}(\lambda \vert u-v\vert ) \right]^{k} dvdu,$$
\end{prop}
where $\lambda >0$, $d >\dfrac{1}{2} - \dfrac{1}{2k}$ (equivalently $H>1/2$) and $K_{v}(x)$ is a modified Bessel function of the second kind (see \cite{bowman, embrech} for details).

\begin{prop}\label{spec-hermite}
Let $Z^{\lambda}$ be the process given by (\ref{hermite-1}) has the spectral domain representation
$$Z^{\lambda}(t) = C_{d,k} \int_{\mathbb{R}^{k}}^{"} \dfrac{e^{it(\omega_{1} + \ldots + \omega_{k})}-1}{i(\omega_{1} + \ldots + \omega_{k})} \prod_{j=1}^{k} (\lambda + i \omega_{j})^{-d} \widehat{W}(dx_{1}) \ldots \widehat{W}(dx_{k}), $$
\end{prop}

Now, we give the expression for the cumulants for the process when $k=2$. It is known that,  for a second chaos process the law of the process is completely determined by their cumulants (see  \cite{fox}). In fact,  if we consider a multiple integral $I_{2}(f)$ of order two with $f \in L^{2}(\mathbb{R}^2)$ symmetric. Then
the m-th cumulant of the random variable $I_{2}(f)$ are given by (see \cite{nou-0})
\begin{equation}\label{cumulant}
C_{m}(I_{2}(f)) = 2^{m-1} (m-1)! \int_{\mathbb{R}^m} f(x_{1},x_{2}) f(x_{2},x_{3})  \cdots f(x_{m-1},x_{m}) f(x_{m},x_{1}) dx_{1}\cdots dx_{m}. 
\end{equation}
Also, we will need the following formula (see \cite{sab}  for the details)
 \begin{lemma}\label{int}
 Let  $\tau \in (0,1/2)$ and $\lambda >0$. Then, 
 \begin{equation*}
 \int_{\mathbb{R}} e^{-\lambda(u-x)_{+}}e^{-\lambda(v-x)_{+}}(u-x)^{\tau-1}_{+}(v-x)^{\tau-1}_{+} dx = (2\lambda)^{\frac{1}{2} -\tau}\dfrac{\Gamma(\tau)}{\sqrt{\pi}} K_{\dfrac{1}{2}-\tau}(\lambda\vert u-v\vert) \vert u-v\vert^{\tau-1/2},
 \end{equation*}
 where $K_{v}(x)$ is a modified Bessel function of the second kind.
 \end{lemma}
Using formula  \eqref{cumulant} and Lemma \ref{int}, we can obtain the following result concerning the cumulants of the process $Z^{\lambda}$ with $k=2$. 
\begin{lemma}\label{approximacion}
Let $Z^{\lambda}$ the process given by  \eqref{process} with $k=2$, $H>1/2$ and $\lambda >0$, and 
\begin{equation}
 h_{t}^{\lambda}(x_{1}, x_{2}) = \int_{0}^{t} e^{-\lambda(s-x_{1})_{+}} e^{-\lambda(s-x_{2})_{+}} (s-x_{1})_{+}^{d-1} (s-x_{2})_{+}^{d-1} ds, \label{kernel}
 \end{equation}
then 
\begin{eqnarray}
C_{m}(Z^{\lambda}(t)) &=& 2^{m-1+ (\frac{1}{2} -d)m  }  (m-1)!  \lambda^{(\frac{1}{2} -d)m} \left( \frac{\Gamma(d)}{\sqrt{\pi}} \right)^m  \int_{0}^{t} \dots \int_{0}^{t}  ds_{1}\cdots ds_{m}  \nonumber \\
& \times &  K_{(\frac{1}{2} -d)}(\lambda\vert s_{1}- s_{2}\vert) \vert s_{1}- s_{2}\vert^{d-1/2}  K_{(\frac{1}{2} -d)}(\lambda\vert s_{2}- s_{3}\vert) \vert s_{2}- s_{3}\vert^{d-1/2} \nonumber \\
& \cdots & K_{(\frac{1}{2} -d)}(\lambda\vert s_{m}- s_{1}\vert) \vert s_{m}- s_{1}\vert^{d-1/2}. \label{cum-1}
\end{eqnarray}
\end{lemma}
\begin{proof}
By \eqref{process} and \eqref{kernel} we have that, for $k=2$, we can write 
\begin{eqnarray*}\label{ros-1}
Z^{\lambda}(t) &=& \int_{\mathbb{R}^{2}}^{'} \int_{0}^{t} \prod_{i=1}^{2} (s-x_{i})^{d-1}_{+} e^{-\lambda (s-x_{i})_{+}}ds  W(dx_{1}) \ldots W(dx_{2}) \\
&=&    \int_{\mathbb{R}^{2}}^{'}  h_{t}^{\lambda}(x_{1}, x_{2})  W(dx_{1}) \ldots W(dx_{2}).
\end{eqnarray*}
Now, using the formula \eqref{cumulant} we can obtain 
\begin{eqnarray*}
C_{m}(Z^{\lambda}(t)) &=& 2^{m-1} (m-1)!  \int_{\mathbb{R}^m}  h_{t}^{\lambda}(x_{1}, x_{2}) h_{t}^{\lambda}(x_{2}, x_{3}) \cdots  h_{t}^{\lambda}(x_{m-1}, x_{m})h_{t}^{\lambda}(x_{m}, x_{1}) dx_{1}\cdots dx_{m}\\
&=& 2^{m-1} (m-1)!  \int_{\mathbb{R}^m} dx_{1}\cdots dx_{m} \\
& \times & \left( \int_{0}^{t} e^{-\lambda(s_{1}-x_{1})_{+}} e^{-\lambda(s_{1}-x_{2})_{+}} (s_{1}-x_{1})_{+}^{d-1} (s_{1}-x_{2})_{+}^{d-1} ds_{1} \right)\\
& \times & \left( \int_{0}^{t}  e^{-\lambda(s_{2}-x_{2})_{+}} e^{-\lambda(s_{2}-x_{3})_{+}}  (s_{2}-x_{2})_{+}^{d-1} (s_{2}-x_{3})_{+}^{d-1}ds_{2} \right) \\
&\vdots & \\
& \times & \left( \int_{0}^{t}  e^{-\lambda(s_{m}-x_{m})_{+}} e^{-\lambda(s_{m}-x_{1})_{+}}  (s_{m}-x_{m})_{+}^{d-1} (s_{m}-x_{1})_{+}^{d-1}ds_{m} \right). 
\end{eqnarray*}
Then, by Fubbini theorem, we can get 
\begin{eqnarray*}
C_{m}(Z^{\lambda}(t)) &=& 2^{m-1} (m-1)!  \int_{0}^{t} \dots \int_{0}^{t}  ds_{1}\cdots ds_{m} \\
& \times & \left( \int_{\mathbb{R}} e^{-\lambda(s_{1}-x_{1})_{+}} e^{-\lambda(s_{m}-x_{1})_{+}} (s_{1}-x_{1})_{+}^{d-1} (s_{m}-x_{1})_{+}^{d-1} dx_{1} \right)\\
&\vdots & \\
& \times & \left( \int_{\mathbb{R}}  e^{-\lambda(s_{m-1}-x_{m})_{+}} e^{-\lambda(s_{m}-x_{m})_{+}}  (s_{m}-x_{m})_{+}^{d-1} (s_{m-1}-x_{m})_{+}^{d-1}dx_{m} \right). 
\end{eqnarray*}
Using Lemma \ref{int}, we can obtain 
\begin{eqnarray*}
C_{m}(Z^{\lambda}(t)) &=& 2^{m-1+ (\frac{1}{2} -d)m  }  (m-1)!  \lambda^{(\frac{1}{2} -d)m} \left( \frac{\Gamma(d)}{\sqrt{\pi}} \right)^m  \int_{0}^{t} \dots \int_{0}^{t}  ds_{1}\cdots ds_{m}  \nonumber \\
& \times &  K_{(\frac{1}{2} -d)}(\lambda\vert s_{1}- s_{2}\vert) \vert s_{1}- s_{2}\vert^{d-1/2}  K_{(\frac{1}{2} -d)}(\lambda\vert s_{2}- s_{3}\vert) \vert s_{2}- s_{3}\vert^{d-1/2} \nonumber \\
& \cdots & K_{(\frac{1}{2} -d)}(\lambda\vert s_{m}- s_{1}\vert) \vert s_{m}- s_{1}\vert^{d-1/2}. \label{cum-2}
\end{eqnarray*}
\end{proof}

\begin{remark}
By taking m=2 in Formula \eqref{cum-1} we can recover the formula for the variance of the tempered Hermite process with $k=2$. 
\end{remark}

Now, we present an interesting result related to the behavior of the process as $\lambda \rightarrow 0^{+}$

\begin{lemma}\label{limite-cumulantes}
Let $Z^{\lambda}$ be the process given by  \eqref{process} with $k=2$, $H>1/2$ and $\lambda >0$, and 
\begin{equation*}
 h_{t}^{\lambda}(x_{1}, x_{2}) = \int_{0}^{t} e^{-\lambda(s-x_{1})_{+}} e^{-\lambda(s-x_{2})_{+}} (s-x_{1})_{+}^{d-1} (s-x_{2})_{+}^{d-1} ds, \label{kernel}
 \end{equation*}
then 
$$ \lim\limits_{\lambda \rightarrow 0^{+}} Z^{\lambda}_{t} \overset{d}{=} Z_{t},$$
where $Z$ is the Rosenblatt process. 
\end{lemma}
\begin{proof}
Let us consider $b_{1}, \ldots, b_{n} \in \mathbb{R}$ and $t_{1}, \ldots , t_{n} \in (0, \infty)$. We need to show that the random variables 
$$ \lim\limits_{\lambda \rightarrow 0^{+}} \sum_{l=1}^{n} b_{l }Z^{\lambda}_{t_{l}} \qquad ; \qquad \sum_{l=1}^{n} b_{l }Z_{t_{l}}$$
have the same distribution. To do this, we will use the cumulant criterium (see Formula \eqref{cumulant}). Also, to simplified computations we will study the limit when $\lambda \rightarrow 0^{+}$ of the cumulants of $Z_{t}^{\lambda} + Z_{s}^{\lambda}$;  the general case follows by similar arguments.\\

We have that 
$$Z_{t}^{\lambda} + Z_{s}^{\lambda} = I_{2}(h^{\lambda}_{t,s}),$$
where 
\begin{eqnarray*}
h^{\lambda}_{t,s} &=& \int_{0}^{t} e^{-\lambda(u-x_{1})_{+}} e^{-\lambda(u-x_{2})_{+}} (u-x_{1})_{+}^{d-1} (u-x_{2})_{+}^{d-1} du \\
&+& \int_{0}^{s} e^{-\lambda(u-x_{1})_{+}} e^{-\lambda(u-x_{2})_{+}} (u-x_{1})_{+}^{d-1} (u-x_{2})_{+}^{d-1} du.
\end{eqnarray*} 
By Formula \eqref{cumulant}
\begin{eqnarray*}
C_{m}(Z_{t}^{\lambda} + Z_{s}^{\lambda}) &=& 2^{m-1} (m-1)!  \int_{\mathbb{R}^m}  h_{t,s}^{\lambda}(x_{1}, x_{2}) h_{t,s}^{\lambda}(x_{2}, x_{3}) \cdots  h_{t,s}^{\lambda}(x_{m-1}, x_{m})h_{t,s}^{\lambda}(x_{m}, x_{1}) dx_{1}\cdots dx_{m}\\
&=& 2^{m-1} (m-1)!  \int_{\mathbb{R}^m} dx_{1}\cdots dx_{m} \\
&\times& \left( \int_{0}^{t} e^{-\lambda(u_1-x_{1})_{+}} e^{-\lambda(u_1-x_{2})_{+}} (u_1 -x_{1})_{+}^{d-1} (u_1-x_{2})_{+}^{d-1} du_1 \right. \\
&+& \left. \int_{0}^{s} e^{-\lambda(u_1-x_{1})_{+}} e^{-\lambda(u_1-x_{2})_{+}} (u_1-x_{1})_{+}^{d-1} (u_1-x_{2})_{+}^{d-1} du_1 \right)\\
&\times& \left( \int_{0}^{t} e^{-\lambda(u_2-x_{2})_{+}} e^{-\lambda(u_2-x_{3})_{+}} (u_2 -x_{2})_{+}^{d-1} (u_2-x_{3})_{+}^{d-1} du_2 \right. \\
&+& \left. \int_{0}^{s} e^{-\lambda(u_2-x_{2})_{+}} e^{-\lambda(u_2-x_{3})_{+}} (u_2-x_{2})_{+}^{d-1} (u_2-x_{3})_{+}^{d-1} du_2 \right)\\
&\vdots & \\
&\times& \left( \int_{0}^{t} e^{-\lambda(u_m-x_{m})_{+}} e^{-\lambda(u_m-x_{1})_{+}} (u_m -x_{m})_{+}^{d-1} (u_m-x_{1})_{+}^{d-1} du_m \right. \\
&+& \left. \int_{0}^{s} e^{-\lambda(u_m-x_{m})_{+}} e^{-\lambda(u_m-x_{1})_{+}} (u_m-x_{m})_{+}^{d-1} (u_m-x_{1})_{+}^{d-1} du_m \right).
\end{eqnarray*}
By Fubbini theorem, we can get
\begin{eqnarray*}
C_{m}(Z^{\lambda}(t)) &=& 2^{m-1} (m-1)!  \sum_{t_{j} \in \{ t,s \}} \int_{0}^{t_1} \dots \int_{0}^{t_m}  du_{1}\cdots du_{m} \\
& \times & \left( \int_{\mathbb{R}} e^{-\lambda(u_{1}-x_{1})_{+}} e^{-\lambda(u_{m}-x_{1})_{+}} (u_{1}-x_{1})_{+}^{d-1} (u_{m}-x_{1})_{+}^{d-1} dx_{1} \right)\\
&\vdots & \\
& \times & \left( \int_{\mathbb{R}}  e^{-\lambda(u_{m-1}-x_{m})_{+}} e^{-\lambda(u_{m}-x_{m})_{+}}  (u_{m}-x_{m})_{+}^{d-1} (u_{m-1}-x_{m})_{+}^{d-1}dx_{m} \right). 
\end{eqnarray*}
As before, using Lemma \ref{int}, we can obtain 
\begin{eqnarray*}
C_{m}(Z_{t}^{\lambda} + Z_{s}^{\lambda}) &=& 2^{m-1+ (\frac{1}{2} -d)m  }  (m-1)!  \lambda^{(\frac{1}{2} -d)m} \left( \frac{\Gamma(d)}{\sqrt{\pi}} \right)^m   \sum_{t_{j} \in \{ t,s \}} \int_{0}^{t_1} \dots \int_{0}^{t_m}  du_{1}\cdots du_{m}  \nonumber \\
& \times &  K_{(\frac{1}{2} -d)}(\lambda\vert u_{1}- u_{2}\vert) \vert u_{1}- u_{2}\vert^{d-1/2}  K_{(\frac{1}{2} -d)}(\lambda\vert u_{2}- u_{3}\vert) \vert u_{2}- u_{3}\vert^{d-1/2} \nonumber \\
& \cdots & K_{(\frac{1}{2} -d)}(\lambda\vert u_{m}- u_{1}\vert) \vert u_{m}- u_{1}\vert^{d-1/2}. \label{cum-2}
\end{eqnarray*}
Now, we use that the function $K_{v}$ is continuous and, for any $v \in \mathbb{R}$, it satisfies as $u \rightarrow 0^{+}$
$$ K_v=  \left\{ \begin{array}{lcc}
             2^{\vert v \vert -1} \Gamma(\vert v\vert )  u ^{-\vert v \vert}&   if  & v \neq 0\\
             \\ - \log u &  if & v=0, 
             \end{array}
   \right.
$$ 
with this, we can obtain 
\begin{eqnarray}
\lim\limits_{\lambda \rightarrow 0^{+}}  C_{m}(Z_{t}^{\lambda} + Z_{s}^{\lambda}) &=&  a(m)  \sum_{t_{j} \in \{ t,s \}} \int_{0}^{t_1} \dots \int_{0}^{t_m}  du_{1}\cdots du_{m}  \nonumber \\
& \times &   \vert s_{1}- s_{2}\vert^{2d-1}  \vert s_{2}- s_{3}\vert^{2d-1}  \cdots   \vert s_{m}- s_{1}\vert^{2d-1}, \label{cum-2-0}
\end{eqnarray}
where $a(m) = 2^{-1}  (m-1)!  \left( \frac{\Gamma(d) \Gamma(d-1/2)}{\sqrt{\pi}} \right)^m $. To conclude, we  compare Formula \eqref{cum-2-0} with the Formula (54) in \cite{tudor-0}.
\end{proof}

\begin{remark}
Here, we assumed that, if two processes differs by a constant, they have the same distribution. 
\end{remark}

\subsection{Hermite kernel: filtered version}
In this part, we consider the filtered Hermite kernel given by 
$$g({\bf x}) = l^{\beta}_{t}(u) \prod_{j=1}^{k}  x_{j}^{d-1}, \quad x_{j}>0 \qquad \mbox{and} \quad u \in \mathbb{R}.$$
 This kernel fulfills the conditions of Definition \ref{hermite-kernel}, this comes by following the lines of Proposition \ref{finite-2} in Section \ref{process-filtered} and Lemma 1 in \cite{sab}. Using this kernel, we define the following process
\begin{equation}\label{hermite-2}
Z^{\lambda}(t) = \int_{\mathbb{R}^{k}}^{'} \int_{\mathbb{R}} l^{\beta}_{t}(u) \prod_{i=1}^{k} (u-x_{i})^{d-1}_{+} e^{-\lambda (u-x_{i})_{+}}du  W(dx_{1}) \ldots W(dx_{k})  ,  
\end{equation}
where $\lambda >0$, $(x)_{+} = x1_{\{ x>0 \}}$ and $d \in \left( \dfrac{1}{2} - \dfrac{1}{2k}, \infty \right)$, and
$$-1 < H < \beta < 1-H < \dfrac{1}{2}, \quad \beta \neq 0.$$
Recall that, the function $l^{\beta}_{t}$ is given by 
$$l_{t}^{\beta}(u) = (t-u)^{\beta}_{+} - (-u)^{\beta}_{+} \quad \mbox{if} \quad \beta \neq 0.$$
and $l_{t}^{\beta}(u) =1_{[0,t]}(u)$ if $\beta =0$.

\begin{remark}
If $\beta =0$ we recover the definition of the tempered Hermite process, and if the tempering factor is one with  $\beta \neq 0$ we have a special case of generalized Hermite process introduced in \cite{bai} and further studied in \cite{Ass}.
\end{remark}

By Proposition \ref{finite-2} and Lemma \ref{sssi-2}  we know that $Z^{\lambda}$ is a stationary increment process with a scaling property. Now, we consider the computation of the explicit expression for the covariance of the process 

\begin{prop}\label{cov-hermite2}
Let $Z^{\lambda}$ be the process given by (\ref{filter-hermite}) with $g(x) = (x)_{+}^{d-1}$. Then, $Z^{\lambda}$ has the covariance function 
$$E(Z^{\lambda}(t) Z^{\lambda}(s)) = \dfrac{1}{2} \left[C_{\vert t \vert} \vert t \vert^{2\beta + k(d-1/2)} + C_{\vert s \vert} \vert s \vert^{2\beta + k(d-1/2)} - C_{\vert t -s\vert} \vert t-s \vert ^{2\beta + k(d-1/2)} \right],$$ 
 \begin{eqnarray*}
C_{\vert t \vert } &:= & k! \left( \dfrac{\Gamma(d)}{\sqrt{\pi} (2\lambda)^{d-1/2}} \right)^k \\
&\times &\int_{\mathbb{R}} \int_{\mathbb{R}}   [ (1-u)^{\beta}_{+} - (-u)^{\beta}_{+}][ (1-v)^{\beta}_{+} - (-v)^{\beta}_{+} ]\vert u-v \vert^{k(d -1/2)} K^k_{1/2 -d}(\lambda t \vert u-v \vert)   dv du.
\end{eqnarray*}
\end{prop}
\begin{proof}
By Lemma \ref{cov-fill}
\begin{eqnarray*}
E(Z^{\lambda}(t) Z^{\lambda}(s) ) &=&  k!  \int_{\mathbb{R}} \int_{\mathbb{R}}  e^{-\lambda k \vert u-v \vert} \vert u-v \vert^{k(2d -1)}  l^{\beta}_{s}(u) l^{\beta}_{t}(v)  \left[ \int_{0}^{\infty} g[x] g[1+x] e^{-2\lambda \vert u-v \vert x}    dx \right]^{k}   dv du,
\end{eqnarray*}
now, taking $g(x)=(x)^{d-1}_{+}$ and the fact that 
$$\int_{0}^{\infty} x^{d-1} (1+x)^{d-1} e^{-2\lambda \vert u-v \vert x}    dx= \dfrac{\Gamma(d)}{\sqrt{\pi} (2\lambda)^{d-1/2}} \vert u-v\vert^{1/2 -d}  e^{\lambda \vert u-v \vert} K_{1/2 -d}(\lambda \vert u-v \vert).$$
With this, we can have
\begin{eqnarray*}
E(Z^{\lambda}(t) Z^{\lambda}(s) ) &=&  k! \left( \dfrac{\Gamma(d)}{\sqrt{\pi} (2\lambda)^{d-1/2}} \right)^k \int_{\mathbb{R}} \int_{\mathbb{R}}    l^{\beta}_{s}(u) l^{\beta}_{t}(v) \vert u-v \vert^{k(d -1/2)} K^k_{1/2 -d}(\lambda \vert u-v \vert)   dv du,
\end{eqnarray*}
let us recall that $l_{t}^{\beta}(u) = (t-u)^{\beta}_{+} - (-u)^{\beta}_{+} \quad \mbox{if} \quad \beta \neq 0.$ Therefore, 
\begin{small}
\begin{eqnarray}
E(Z^{\lambda}(t) Z^{\lambda}(s) ) &=&  k! \left( \dfrac{\Gamma(d)}{\sqrt{\pi} (2\lambda)^{d-1/2}} \right)^k \nonumber \\
&\times&\int_{\mathbb{R}} \int_{\mathbb{R}}   [ (s-u)^{\beta}_{+} - (-u)^{\beta}_{+}][ (t-v)^{\beta}_{+} - (-v)^{\beta}_{+} ]\vert u-v \vert^{k(d -1/2)} K^k_{1/2 -d}(\lambda \vert u-v \vert)   dv du. \nonumber \\ \label{cov-filter}
\end{eqnarray}
\end{small}
To take advantage of the expression \eqref{cov-filter}, we will  use the fact that 
$$E(Z^{\lambda}(t) Z^{\lambda}(s)) = \dfrac{1}{2} \left[E([Z^{\lambda}(t)]^2) + E([Z^{\lambda}(s)]^2) - E([Z^{\lambda}(t)-Z^{\lambda}(s)]^2) \right].$$ 
Now, we concentrate in the computation of $E([Z^{\lambda}(t)]^2)$. To do this, we make the successive change of variables $\tilde{u} = u/t$ and $\tilde{v} = v/t$
\begin{eqnarray*}
E(Z^{\lambda}(t)^2 ) &=&  k! \left( \dfrac{\Gamma(d)}{\sqrt{\pi} (2\lambda)^{d-1/2}} \right)^k \vert t \vert ^{2\beta + k(d-1/2)}\nonumber \\
&\times&\int_{\mathbb{R}} \int_{\mathbb{R}}   [ (1-u)^{\beta}_{+} - (-u)^{\beta}_{+}][ (1-v)^{\beta}_{+} - (-v)^{\beta}_{+} ]\vert u-v \vert^{k(d -1/2)} K^k_{1/2 -d}(\lambda t \vert u-v \vert)   dv du. \nonumber \\ \label{cov-filter-2}
\end{eqnarray*}
If we define 
\begin{eqnarray*}
C_{\vert t \vert } &:= & k! \left( \dfrac{\Gamma(d)}{\sqrt{\pi} (2\lambda)^{d-1/2}} \right)^k \\
&\times &\int_{\mathbb{R}} \int_{\mathbb{R}}   [ (1-u)^{\beta}_{+} - (-u)^{\beta}_{+}][ (1-v)^{\beta}_{+} - (-v)^{\beta}_{+} ]\vert u-v \vert^{k(d -1/2)} K^k_{1/2 -d}(\lambda t \vert u-v \vert)   dv du.
\end{eqnarray*}
Then,
\begin{eqnarray*}
E(Z^{\lambda}(t)^2 ) &=&  C_{\vert t \vert } \vert t \vert^{2\beta + k(d-1/2)}.
\end{eqnarray*}
With this and using the stationarity of the process $Z^{\lambda}$ we can get 
$$E(Z^{\lambda}(t) Z^{\lambda}(s)) = \dfrac{1}{2} \left[C_{\vert t \vert} \vert t \vert^{2\beta + k(d-1/2)} + C_{\vert s \vert} \vert s \vert^{2\beta + k(d-1/2)} - C_{\vert t -s\vert} \vert t-s \vert ^{2\beta + k(d-1/2)} \right].$$ 
\end{proof}
As before, using Formula  \eqref{cumulant}, we can obtain the following result concerning the cumulants of the process $Z^{\lambda}$  with $k=2$ in the filtered case. 
\begin{lemma}\label{approximacion-2}
Let $Z^{\lambda}$ the process given by  (\ref{filter-hermite}) with $g(x) = (x)_{+}^{d-1}$, $k=2$; $\lambda >0$, and 
\begin{equation*}
 h_{t}^{\lambda}(x_{1}, x_{2}) = \int_{\mathbb{R}} l^{\beta}_{t}(u)  (u-x_{1})^{d-1}_{+} e^{-\lambda (u-x_{1})_{+}}(u-x_{2})^{d-1}_{+} e^{-\lambda (u-x_{2})_{+}}du, \label{kernel-2}
 \end{equation*}
then 
\begin{eqnarray*}
C_{m}(Z^{\lambda}(t)) &=& 2^{m-1+ (\frac{1}{2} -d)m  }  (m-1)!  \lambda^{(\frac{1}{2} -d)m} \left( \frac{\Gamma(d)}{\sqrt{\pi}} \right)^m  \int_{\mathbb{R}} \dots \int_{\mathbb{R}} du_{1}\cdots du_{m}  \prod_{j=1}^{m} l_{t}^{\beta}(u_{j})  \nonumber  \\
& \times &  K_{(\frac{1}{2} -d)}(\lambda\vert u_{1}- u_{2}\vert) \vert u_{1}- u_{2}\vert^{d-1/2}  K_{(\frac{1}{2} -d)}(\lambda\vert u_{2}- u_{3}\vert) \vert u_{2}- u_{3}\vert^{d-1/2} \nonumber \\
& \cdots & K_{(\frac{1}{2} -d)}(\lambda\vert u_{m}- u_{1}\vert) \vert u_{m}- u_{1}\vert^{d-1/2}. \label{cum-2}
\end{eqnarray*}
\end{lemma}
\begin{proof}
By \eqref{hermite-2} we have that, for $k=2$ we can write 
\begin{eqnarray*}\label{ros-1}
Z^{\lambda}(t) &=& \int_{\mathbb{R}^{2}}^{'} \int_{\mathbb{R}} l^{\beta}_{t}(u) \prod_{i=1}^{2}  (u-x_{i})^{d-1}_{+} e^{-\lambda (u-x_{i})_{+}}du  W(dx_{1})  W(dx_{2}) \\
&=&    \int_{\mathbb{R}^{2}}^{'}  h_{t}^{\lambda}(x_{1}, x_{2})  W(dx_{1}) \ldots W(dx_{2}).
\end{eqnarray*}
Now, using the Formula \eqref{cumulant} we can obtain 
\begin{eqnarray*}
C_{m}(Z^{\lambda}(t)) &=& 2^{m-1} (m-1)  \int_{\mathbb{R}^m}  h_{t}^{\lambda}(x_{1}, x_{2}) h_{t}^{\lambda}(x_{2}, x_{3}) \cdots  h_{t}^{\lambda}(x_{m-1}, x_{m})h_{t}^{\lambda}(x_{m}, x_{1}) dx_{1}\cdots dx_{m}\\
&=& 2^{m-1} (m-1)  \int_{\mathbb{R}^m} dx_{1}\cdots dx_{m} \\
& \times & \left( \int_{\mathbb{R}} l^{\beta}_{t}(u_1) e^{-\lambda(u_{1}-x_{1})_{+}} e^{-\lambda(u_{1}-x_{2})_{+}} (u_{1}-x_{1})_{+}^{d-1} (u_{1}-x_{2})_{+}^{d-1} du_{1} \right)\\
& \times & \left( \int_{\mathbb{R}} l^{\beta}_{t}(u_2) e^{-\lambda(u_{2}-x_{2})_{+}} e^{-\lambda(u_{2}-x_{3})_{+}}  (u_{2}-x_{2})_{+}^{d-1} (u_{2}-x_{3})_{+}^{d-1}du_{2} \right) \\
&\vdots & \\
& \times & \left( \int_{\mathbb{R}} l^{\beta}_{t}(u_m) e^{-\lambda(u_{m}-x_{m})_{+}} e^{-\lambda(u_{m}-x_{1})_{+}}  (u_{m}-x_{m})_{+}^{d-1} (u_{m}-x_{1})_{+}^{d-1}ds_{m} \right). 
\end{eqnarray*}
Then, by Fubbini theorem, we can obtain 
\begin{eqnarray*}
C_{m}(Z^{\lambda}(t)) &=& 2^{m-1} (m-1)!  \int_{\mathbb{R}} \dots \int_{\mathbb{R}} du_{1}\cdots du_{m}  \prod_{j=1}^{m} l_{t}^{\beta}(u_{j}) \\
& \times & \left( \int_{\mathbb{R}} e^{-\lambda(u_{1}-x_{1})_{+}} e^{-\lambda(u_{m}-x_{1})_{+}} (u_{1}-x_{1})_{+}^{d-1} (u_{m}-x_{1})_{+}^{d-1} dx_{1} \right)\\
&\vdots & \\
& \times & \left( \int_{\mathbb{R}}  e^{-\lambda(u_{m-1}-x_{m})_{+}} e^{-\lambda(u_{m}-x_{m})_{+}}  (u_{m}-x_{m})_{+}^{d-1} (u_{m-1}-x_{m})_{+}^{d-1}dx_{m} \right). 
\end{eqnarray*}
Using Lemma \ref{int}, again, we can obtain 
\begin{eqnarray}
C_{m}(Z^{\lambda}(t)) &=& 2^{m-1+ (\frac{1}{2} -d)m  }  (m-1)!  \lambda^{(\frac{1}{2} -d)m} \left( \frac{\Gamma(d)}{\sqrt{\pi}} \right)^m  \int_{\mathbb{R}} \dots \int_{\mathbb{R}} du_{1}\cdots du_{m}  \prod_{j=1}^{m} l_{t}^{\beta}(u_{j})  \nonumber  \\
& \times &  K_{(\frac{1}{2} -d)}(\lambda\vert u_{1}- u_{2}\vert) \vert u_{1}- u_{2}\vert^{d-1/2}  K_{(\frac{1}{2} -d)}(\lambda\vert u_{2}- u_{3}\vert) \vert u_{2}- u_{3}\vert^{d-1/2} \nonumber \\
& \cdots & K_{(\frac{1}{2} -d)}(\lambda\vert u_{m}- u_{1}\vert) \vert u_{m}- u_{1}\vert^{d-1/2}. \label{cum-2}
\end{eqnarray}

\end{proof}

\begin{remark}
As in the tempered Hermite case, taking m=2 in the formula \eqref{cum-2} we can recover the formula for the variance of the tempered filtered Hermite process with $k=2$ (see Equality \eqref{cov-filter}). 
\end{remark}

In a similar manner to the tempered Rosenblatt process, we have the following result concerning the behavior of the tempered Rosenblatt process with filtered kernel when $\lambda \rightarrow 0^{+}$

\begin{lemma}\label{limite-cumulantes-2}
Let $Z^{\lambda}$ be given by (\ref{filter-hermite}) with $g(x) = (x)_{+}^{d-1}$, $k=2$, and $\lambda >0$, and 
\begin{equation*}
 h_{t}^{\lambda}(x_{1}, x_{2}) = \int_{\mathbb{R}} l^{\beta}_{t}(u)  (u-x_{1})^{d-1}_{+} e^{-\lambda (u-x_{1})_{+}}(u-x_{2})^{d-1}_{+} e^{-\lambda (u-x_{2})_{+}}du, \label{kernel-2}
 \end{equation*}
then 
$$ \lim\limits_{\lambda \rightarrow 0^{+}} Z^{\lambda}_{t} \overset{d}{=} Z_{t},$$
where $Z$ is the Rosenblatt process with filtered kernel. 
\end{lemma}
\begin{proof}
The proof follows the same lines of the proof of Lemma \ref{limite-cumulantes}, so we will omit it. 
\end{proof}

\section{An application: non parametric estimation}
Here, we consider the problem of non parametric estimation. Precisely, we consider a cointegrated regressor model where the regressor is a fractional Brownian motion with Hurst parameter $H_1 \in (0,1)$ and $Z^{\lambda}$ is a generalized tempered Hermite process. \\

Let us consider the following model 
\begin{equation}\label{model-non}
Y_{i/n} = r(B^{H_1}_{i/n}) + S_{n}(Z^{\lambda}_{i+1/n} - Z^{\lambda}_{i/n}), \quad 0 \leq i \leq n-1 \quad \mbox{and} \quad n \geq 1,  
\end{equation}
with, as mentioned before, $B^{H_1}$ is a fractional Brownian motion  and $Z^{\lambda}$ is a generalized tempered Hermite process. Here,
$$S_{n} = \sqrt{Var(Z^{\lambda}_{i+1/n} - Z^{\lambda}_{i/n})}.$$
As usual, the estimator of the function $r$ can be written as 
\begin{equation}
\hat{r}_{n}(x) = \dfrac{\sum_{i=0}^{n-1} Y_{i/n} K\left( \dfrac{x-B_{i/n}^{H_1}}{h} \right)}{\sum_{i=0}^{n-1} K\left( \dfrac{x-B_{i/n}^{H_1}}{h} \right)},
\end{equation}
by means of the expression \eqref{model-non} we can decomposed $\hat{r}_{n}$ as 
\begin{eqnarray}
\hat{r}_{n}(x) &=& \dfrac{\sum_{i=0}^{n-1}  K\left( \dfrac{x-B_{i/n}^{H_1}}{h} \right) r(B_{i/n}^{H_1} )}{\sum_{i=0}^{n-1} K\left( \dfrac{x-B_{i/n}^{H_1}}{h} \right)} \nonumber \\
&+& \dfrac{\sum_{i=0}^{n-1}  K\left( \dfrac{x-B_{i/n}^{H_1}}{h} \right) S_{n}(Z^{\lambda}_{i+1/n} - Z^{\lambda}_{i/n})}{\sum_{i=0}^{n-1} K\left( \dfrac{x-B_{i/n}^{H_1}}{h} \right)} \nonumber \\
&:=& M_{1}^{(n)} + M_{2}^{(n)} \label{ms}
\end{eqnarray}
for every $x \in \mathbb{R}$. Where, $K$ is a non-negative real kernel function satisfying $\int_{\mathbb{R}} K(y) dy =1$, $\int_{\mathbb{R}} y K(y) dy =0$ and $\int_{\mathbb{R}} K(y)^2 dy < \infty$. The bandwith $h \equiv h_{n}$ satisfies $h_{n} \rightarrow 0$ as $n \rightarrow \infty$ and 
$$h_{n} := h=n^{-\kappa} \qquad \mbox{with} \qquad 0<\kappa < 1.$$
Now, the idea (as in \cite{sued}) is to prove that the estimator $\hat{r}_{n}$  is consistent, that is, $\hat{r}_{n}$  converges in probability to $r(x) \ \forall x \in \mathbb{R}$.  To prove this, we will handle the terms $M_1$ and $M_2$ by separate. In fact, we will prove that $M_{1}^{(n)}(x) \rightarrow r(x)$ and $M_{2}^{(n)}(x) \rightarrow 0$ as $n \rightarrow \infty$, where both limits are a.s. Then, the final result comes by an application of the continuous mapping theorem. As mentioned before, we will study the terms $M_1$ and $M_2$ by separate. In fact, we have the following result for $M_1$.
\begin{prop}\label{prop-1}
Let assume that $r$ is H\"older continuous with exponent $\gamma_r$. Take $\kappa < \min \left(H_1/2, H_1 \gamma_r \right)$ and for $h=n^{-\kappa}$, let $M_{1}^{(n)}$ be given by \eqref{ms}, with Lipschitz kernel $K$ satisfying conditions (3.7) and (3.9) in \cite{sued}. Then, for every $x \in \mathbb{R}$
$$M_1^{(n)}(x) \rightarrow r(x)$$
as $n \rightarrow \infty$. 
\end{prop}
\begin{proof}
The proof is a consequence of  Lemmas 3.1, 3.3, 3.4 and 3.5 in \cite{sued}.
\end{proof}
\begin{remark}\label{rem-k}
The kernels that satisfies the Lipschitz condition are:
\begin{itemize}
\item The Gaussian kernel 
$$K(x) = \dfrac{1}{\sqrt{2\pi}} e^{-\frac{x^2}{2}}, \qquad x \in \mathbb{R}.$$
\item The triangle kernel 
$$K(x) = (1-\vert x \vert)1_{[-1,1]}(x).$$
\item The Epanechnikov kernel 
$$K(x) = \dfrac{3}{4} (1-  x^2)1_{[-1,1]}(x).$$
\item The quartic kernel 
$$K(x) = \dfrac{15}{16} (1-  x^2)^2 1_{[-1,1]}(x).$$
\end{itemize}
\end{remark}
Now, we need to prove that  $M_{2}^{(n)}(x) \rightarrow 0$ as $n \rightarrow \infty$. To do this we need to study the asymptotic behavior of the term $M_2$. Let us recall that  we can write 
\begin{equation}\label{m2} 
M_{2}^{(n)}(x) = \dfrac{n^{\kappa - 1}\sum_{i=0}^{n-1}  K\left( \dfrac{x-B_{i/n}^{H_1}}{h} \right) S_{n}(Z^{\lambda}_{i+1/n} - Z^{\lambda}_{i/n})}{n^{\kappa - 1} \sum_{i=0}^{n-1} K\left( \dfrac{x-B_{i/n}^{H_1}}{h} \right)} = \dfrac{M_{2.1}^{(n)}(x)}{M_{2.2}^{(n)}(x)}.
\end{equation}
First, we will handle the numerator of the term $M_{2}$. In fact, by Lemma 3 and considering for $i=0$ is easy to see that the convergence holds for $\kappa <1$. We can obtain 
\begin{eqnarray}
E \left[\left(M_{2.1}^{(n)}(x) \right)^2 \right] &\leq &n^{2(\kappa - 1) + 2H}\sum_{i,j=1}^{n-1} E  K\left( \dfrac{x-B_{i/n}^{H_1}}{h} \right)  K\left( \dfrac{x-B_{j/n}^{H_1}}{h} \right) \nonumber \\ 
&\times & E (Z^{\lambda}_{i+1/n} - Z^{\lambda}_{i/n}) (Z^{\lambda}_{j+1/n} - Z^{\lambda}_{j/n}) \nonumber 
\end{eqnarray}
Now, we will consider, only, the Gaussian kernel. Although, the others examples given in Remark \ref{rem-k} can be used without any problem. 
\begin{eqnarray*}
E \left[\left(M_{2.1}^{(n)}(x) \right)^2 \right] &\leq & C n^{2(\kappa - 1)}\sum_{i=1}^{n-1} E  K^{2}\left( \dfrac{x-B_{i/n}^{H_1}}{h} \right)\\
& +& C n^{2(\kappa - 1)}\sum_{i,j=1}^{n-1} E  K\left( \dfrac{x-B_{i/n}^{H_1}}{h} \right)  K\left( \dfrac{x-B_{j/n}^{H_1}}{h} \right) R(i,j) \\
&=& m_{2.1.1}^{(n)}(x) + m_{2.1.2}^{(n)}(x). 
\end{eqnarray*}
For the first term, we can use the Results from Section 3.2.1 in \cite{sued} to get 
\begin{equation}
m_{2.1.1}^{(n)}(x) \leq Cn^{\kappa -1}
\end{equation}
and this always goes to zero for $\kappa <1$. With respect to the second term we must impose the condition that $\vert R(i,j) \vert \leq C \vert i-j\vert^{l_R} $, where $-1<l_{R} <0$, with this and using, again, the results from Section 3.2.1 in \cite{sued}, it allow us to obtain
\begin{equation}
m_{2.1.2}^{(n)}(x) \leq Cn^{2(\kappa -1)  + \frac{4-3\kappa}{2} + l_{R}}
\end{equation}
which converges to zero under the condition $\kappa < -2l_{R}$. Taking into account the previous computations we can have 
\begin{lemma}\label{num-m2}
Let us assume that $K$ is the Gaussian kernel and
$$\kappa < -2l_{R}$$
with $-1<l_{R} <0$ and $M_{2.1}^{(n)}(x)$ given by \eqref{m2}. Then, 
$$ M_{2.1}^{(n)}(x)\longrightarrow 0$$
as $n \rightarrow \infty$ in $L^2(\Omega)$. 
\end{lemma}
Now we consider the behavior of $M_2$.  
\begin{lemma}\label{lemma-m2}
Let us assume that $K$ is the Gaussian kernel and
$$\kappa < \min \{ H_{1}/2, -2l_{R}\}$$
with $-1<l_{R} <0$ and $M_{2}^{(n)}(x)$ given by \eqref{m2}. Then, 
$$ M_{2}^{(n)}(x)\longrightarrow 0$$
as $n \rightarrow \infty$ in probability. 
\end{lemma}
\begin{proof}
The proof follows by Lemma 3.1 in \cite{sued} and Lemma \ref{num-m2}. 
\end{proof}
\vspace{0.5cm}

Finally, we can obtain the following result concerning the convergence of $\hat{r}_{n}$ 
\begin{theorem}
Let us assume that the function $r$ is H\"older continuos with exponent $\gamma_r$ and $K$ is the Gaussian kernel. Also assume that 
$$ \kappa < \min \{H_{1}/2, l_R, H_{1}\gamma_r   \}.  $$
Then, for every $x \in \mathbb{R}$ 
\begin{equation}
\hat{r}(x) \longrightarrow r(x) \nonumber
\end{equation}
as $n\longrightarrow \infty $ in probability.
\end{theorem}
\begin{proof}
The proof follows by Proposition \ref{prop-1}, Lemma \ref{num-m2} and Lemma \ref{lemma-m2}.
\end{proof}
\begin{remark}
The condition $\vert R(i,j) \vert \leq C \vert i-j\vert^{l_R} $, where $-1<l_{R} <0$ is not easy to fulfill, even it seems hard to prove for the simplest non-Gaussian case. At least, we  have that this condition is satisfied in the Gaussian case. However, for the non-Gaussian cases  there is the tempered fractional levy \cite{bon} and the fractional stable \cite{mer} processes  that fulfill this condition. Nonetheless, this process is not in the space of generalized Hermite process.  Although, is a proper non - Gaussian process for our application. The long time  behavior of covariance of the generalized Hermite process remain as an open question and is a topic of future research.
\end{remark}

\section{Appendix: Some elements from Malliavin calculus}
Here, we briefly recall some elements from stochastic analysis; for
an in-depth introduction we refer the reader to \cite{N}.
Consider $({\mathcal{H}},{\langle .,.\rangle}_{\mathcal{H}})$ a real
separable Hilbert space and $(B (\varphi), \varphi\in{\mathcal{H}})$
an isonormal Gaussian process on a probability space $(\Omega,
{\mathfrak{F}}, \mathbb{P})$, which is a centered Gaussian family of
random variables such that $\mathbf{E}\left( B(\varphi) B(\psi)
\right) = {\langle\varphi, \psi\rangle}_{{\mathcal{H}}}$, for every
$\varphi,\psi\in{\mathcal{H}}$. Denote $I_{q}$ the $q$th multiple
stochastic integral with respect to $B$. This $I_{q}$ is actually an
isometry between the Hilbert space ${\mathcal{H}}^{\odot q}$
(symmetric tensor product) equipped with the scaled norm
$\frac{1}{\sqrt{q!}}\Vert\cdot\Vert_{{\mathcal{H}}^{\otimes q}}$ and
the Wiener chaos of order $q$, which is defined as the closed linear
span of the random variables $H_{q}(B(\varphi))$ where
$\varphi\in{\mathcal{H}},\;\Vert\varphi\Vert_{{\mathcal{H}}}=1$ and
$H_{q}$ is the Hermite polynomial of degree $q\geq 1$ defined by:
\begin{equation}
\label{hermitepol}
H_{q}(x)= (-1)^{q} \exp \left( \frac{x^{2}}{2} \right) \frac{d^{q}%
}{dx^{q}}\left( \exp \left( -\frac{x^{2}}{2}\right) \right),
\hskip0.3cm x\in \mathbb{R}.
\end{equation}
The isometry of multiple integrals can be written as: for $p,\;q\geq
1$,\;$f\in{{\mathcal{H}}^{\otimes p}}$ and
$g\in{{\mathcal{H}}^{\otimes q}}$,

\begin{equation} \mathbf{E}\Big(I_{p}(f) I_{q}(g) \Big)= \left\{
\begin{array}{rcl}\label{iso}
q! \langle \tilde{f},\tilde{g}
\rangle _{{\mathcal{H}}^{\otimes q}}&&\mbox{if}\;p=q\\
\noalign{\vskip 2mm} 0 \quad\quad&&\mbox{otherwise}.
\end{array}\right.
\end{equation}

It also holds that:
\begin{equation*}
I_{q}(f) = I_{q}\big( \tilde{f}\big)
\end{equation*}
where $\tilde{f} $ denotes the symmetrization of $f$ defined by $\tilde{f}%
(x_{1}, \ldots , x_{q}) =\frac{1}{q!}\sum_{\sigma\in\mathcal{S}_q}
f(x_{\sigma (1) },\ldots, x_{\sigma (q)}) $.

\section*{Acknowledgements}
The author was partially supported by Proyecto Fondecyt PostDoctorado  3190465, Project ECOS210037, MEC 80190045 and Mathamsud
AMSUD210023.

\end{document}